\documentclass[12pt, reqno]{amsart}
\usepackage{amssymb,dsfont}
\usepackage{eucal}
\usepackage{amsmath}
\usepackage{amscd}
\usepackage[dvips]{color}
\usepackage{multicol}
\usepackage[all]{xy}           %xypic macro for latex2.09
\usepackage{graphicx}
\usepackage{color}
\usepackage{colordvi}
\usepackage{xspace}
\usepackage{txfonts}
\usepackage{lscape}
\usepackage{tikz} % A REMETTRE
\numberwithin{equation}{section}
\usepackage[shortlabels]{enumitem}
\usepackage{ifpdf}
\ifpdf
\usepackage[colorlinks,final,backref=page,hyperindex]{hyperref}
\else
\usepackage[colorlinks,final,backref=page,hyperindex]{hyperref}
\fi
\usepackage{tikz}
\usepackage[active]{srcltx}

\newtheorem{theorem}{Theorem}[section]
\newtheorem{definition}[theorem]{Definition}
\newtheorem{lemma}[theorem]{Lemma}
\newtheorem{corollary}[theorem]{Corollary}
\newtheorem{prop-def}[theorem]{Proposition-Definition}

\newtheorem{remark}[theorem]{Remark}

\numberwithin{equation}{section}

\title{Local derivations and local automorphisms on the super Virasoro algebras}

\author{Qingyan Wu}
\address{College of Mathematics and System Sciences, Xinjiang University, Urumqi 830046, Xinjiang, China}
\email{wuqingyan1019@163.com}

\author{Shoulan Gao}
\address{Department of Mathematics, Huzhou University, Zhejiang Huzhou, 313000, China}
\email{gaoshoulan@zjhu.edu.cn}

\author{Dong Liu}
\address{Department of Mathematics, Huzhou University, Zhejiang Huzhou, 313000, China}
\email{liudong@zjhu.edu.cn}

\author{Chang Ye}
\address{Department of Mathematics, Huzhou University, Zhejiang Huzhou, 313000, China}
\email{yechang@zjhu.edu.cn}

\date{}

\begin{document}

\maketitle

\begin{abstract}
This paper aims to study the local derivations, 2-local automorphisms and local automorphisms on the super-Virasoro algebras. The primary focus is to establish that every local derivation of the super-Virasoro algebras is indeed a derivation, and to demonstrate that every local or 2-local automorphism of the super-Virasoro algebras is an automorphism.
\end{abstract}

{\small \textbf{Keywords}: local derivation; super Virasoro algebra; local automorphism; 2-local automorphism

}

\section{Introduction}

Local derivations are a significant concept for various algebras, which measure some kind of local property of the algebras. A linear mapping $\Delta: L\rightarrow L$ is a local derivation of an algebra $L$ if for every $x\in L$, there exists a derivation $D_{x}: L\rightarrow L$ that depends on $x$ and satisfies $\Delta(x)=D_{x}(x)$. A local derivation is called nontrivial if it is not a derivation. Kadison \cite{Kadison}, Larson and Sourour \cite{LarSou} introduced the concept of local derivation for Banach (or associative) algebras in 1990, which turns out to be very interesting and has been studied in several papers (see \cite{AK,DGL,LiuZhang}). Many researchers have focused on studying local derivations for Lie algebras (see \cite{AK,CZZ,WGL}. Recently, more and more scholars have begun to study local and 2-local derivation on Lie superalgebras. For instance, in \cite{DGL}, it was proved that every 2-local derivation on the super Virasoro algebras and the super $W(2,2)$ is a derivation. However, determining all local derivations on Lie superalgebras requires specialized techniques for each algebra, and there is no uniform method for this purpose.

In last decades a series of papers have been devoted to mappings which are close to automorphism of associative algebras. Namely, the problems of description of so called local automorphisms and 2-local automorphisms have been considered. Let $L$ be an associative algebra. A linear operator $\phi$ on $L$ is called a local automorphism if for every $x\in L$ there exists an automorphism $\theta_{x}$ of $L$, depending on $x$, such that $\phi(x)=\theta_{x}(x)$. The concept of local automorphism was introduced by Larson and Sourour \cite{LarSou} in $1990$. In $1997$, ${\rm \check{S}}$emrl \cite{Sem} introduced the notion of 2-local automorphisms of algebras. Namely, a map $\phi:L\rightarrow L$ (not necessarily linear) is called a 2-local automorphism if for every $x,y\in L$, there exists an automorphism $\theta_{x,y}:L\rightarrow L$ such that $\phi(x)=\theta_{x,y}(x)$ and $\phi(y)=\theta_{x,y}(y)$. These concepts are actually important and interesting properties for an algebra. Recently, several papers have been devoted to similar notions and corresponding problems for Lie(super) algebras $L$ (see \cite{SK,CZY,YU}). The main problem in this subject is to determine all local and 2-local automorphisms, and to see whether every local or 2-local automorphism automatically becomes an automorphism of $L$, that is whether automorphisms of an algebra can be completely determined by their local actions.

In this paper, we investigate local derivations, local and 2-local automorphisms on the super-Virasoro algebras based on the work of \cite{CZZ, CZY} and \cite{WGL}.  Our results show that every local derivation of the super-Virasoro algebras is, in fact, a derivation. And we demonstrate that every local or 2-local automorphism of the super Virasoro algebras is an automorphism.

The structure of this paper is as follows: In Section 2, we review relevant results and establish some related properties. In Sections 3 and 4, we prove that every local derivation of the super Virasoro algebras is a derivation. In Section 5, we prove that every 2-local or local automorphism of the super Virasoro algebras is an automorphism.

Throughout this paper, we use the following notation: $\mathbb{Z}$ and $\mathbb{C}$ denote the sets of integers and complex numbers, respectively. $\mathbb{Z}^*$ and $\mathbb{C}^*$ denote the sets of nonzero integers and nonzero complex numbers, respectively. All algebras are defined over $\mathbb{C}$.

\medskip

\section{Preliminaries}

\medskip

In this section, we recall some definitions, notations and some basic results for later use in this paper.

\subsection{}
Suppose that $L=L_{\overline{0}} \oplus L_{\overline{1}}$  is a superalgebra. If a linear map $f: L\rightarrow L$ such that
 $$
 f(L_i)\subseteq L_{i+\tau}, \, \forall i\in \mathbb{Z}_2,
 $$
then $ f$ is called a homogeneous linear map of degree $\tau$, that is $|f |=\tau$, where $\tau \in\mathbb{Z}_2$.
 If $\tau=\overline{0}$, then $f$ is called an even linear map and if $\tau=\overline{1}$, then $f$ is called an odd linear map.   For a linear map $f$, if $| f |$ occurs, then $f$ implies a homogeneous map. For a Lie superalgebra  $L=L_{\overline{0}} \oplus L_{\overline{1}}$, if $|x|$ occurs, then $x$ implies  a homogeneous element in $L$.

\begin{definition}\label{def1}
Let $L$ be a Lie superalgebra, we call  a linear map $D: L \rightarrow L $ a derivation of $L$ if
\begin{equation*}%\label{(1.9)}
  D([x, y])= [D(x),y]+(-1)^{|D||x|} [x,D(y)], \forall x,y\in L.
\end{equation*}Denote by ${\rm Der}(L)$ the set of all derivations of $L$.
 For any $u\in L$, the map ${\rm ad}(u)$ on $L$ defined as ${\rm ad}(u)(x)=[u, x], x \in L$, is a derivation and derivations of this form are called inner derivations. Denote by ${\rm Inn}(L)$ the set of all inner derivations of $L$.
\end{definition}

 \begin{definition}
 A linear map $\Delta: L\rightarrow L$  is called a local derivation if, for every   $x\in L$, there exists a derivation  $D_{x}: L\rightarrow L$ (depending on $x$)  such that $D_{x}(x)=\Delta(x)$.
 \end{definition}

 \begin{definition}
 A linear map $\phi: L\rightarrow L$ on $L$ is called a local automorphism if, for every   $x\in L$, there exists an automorphism  $\theta_{x}$ of $L$ (depending on $x$)  such that $\theta_{x}(x)=\phi(x)$.
 \end{definition}

 \begin{definition}
 A map $\phi: L\rightarrow L$ on $L$ is called a 2-local automorphism if, for every   $x,y\in L$, there exists an automorphism  $\theta_{x,y}:L\rightarrow L$ such that $\theta_{x,y}(x)=\phi(x)$ and $\theta_{x,y}(y)=\phi(y)$.
 \end{definition}

 \begin{definition}\label{def2}
  For $\epsilon=0, \frac12$,  the super Virasoro algebras ${\rm SVir}[\epsilon]$ are Lie superalgebras with a basis  $ \{ L_{m}, G_{r}, C \mid m\in \mathbb{Z}, r\in\mathbb{Z}+\epsilon\}$, equipped with the following relations:
\begin{eqnarray*}
&&[L_{m}, L_{n}]= (m-n)L_{m+n}+\frac{1}{12}\delta_{m+n, 0}(m^{3}-m)C,\\
&&[L_{m}, G_{r}]= (\frac{m}{2}-r)G_{m+r},\\
&&[G_{r}, G_{s}]= 2L_{r+s}+\frac{1}{3}\delta_{r+s, 0}(r^{2}-\frac{1}{4})C,
\end{eqnarray*}
for all $ m, n \in \mathbb{Z}, \ r, s\in \mathbb{Z}+\epsilon $.
\end{definition}

From the main result in \cite{Zhao}, we can get the following result.

\begin{lemma}\label{lem-local-auto1}
Suppose that  $\sigma\in {\rm Aut\,(SVir}[\epsilon])$. Then
\begin{eqnarray}
\sigma(L_{m})=\varepsilon a^{m}L_{\varepsilon m}\label{aut1},
\sigma(G_{r})=\frac{1}{\sqrt{\varepsilon}} a^{r}G_{\varepsilon r},\sigma(C)=\varepsilon C,
\end{eqnarray}
where $\varepsilon=\pm 1,a\in \mathbb{C}^{*},m\in\mathbb{Z},r\in\mathbb{Z}+\epsilon$.
\end{lemma}

\begin{lemma}\cite{GMP}\label{lem2.4} For the super Virasoro algebras, we have ${\rm Der\,(SVir[\epsilon])}={\rm Inn\,(SVir[\epsilon])}.$
\end{lemma}

\section{Local derivations on  the  Virasoro  subalgebra}

In this section we present the main result on local derivations of the  Virasoro subalgebra of ${\rm SVir}[0]$. The methods are essentially same as that in Section 3 of \cite{CZZ} and \cite{WGL}.

Denoted by $\mathcal{S}:={\rm SVir}[0]/\mathbb C\{C\}$. For a local derivation $\Delta:\mathcal{S} \rightarrow \mathcal{S}$ and $x\in  \mathcal{S}$,  we always use the symbol $D_{x}$ for the derivation of $\mathcal{S}$ satisfying $\Delta(x)=D_{x}(x)$ in the following sections.

For a given $m\in\mathbb{Z}^\ast$, recall that $\mathbb{Z}_m=\mathbb{Z}/m\mathbb{Z}$ is the modulo $m$ residual ring of $\mathbb{Z}$.
Then for any $i\in \mathbb{Z}$ we have $\overline{i}\in \mathbb{Z}_m$ which is the subset of $\mathbb{Z}$ as $$\overline{i} =\{i+km \mid k\in\mathbb{Z}\}.$$

 Let $\Delta$ be a local derivation on $\mathcal{S}$ with $\Delta( L_{0})=0 $. For $L_{m}$ with $m\neq0$, there exists $y\in\mathcal S$ such that
\begin{eqnarray}\label{svir2}
\Delta(L_{m})=[y, L_m]=\sum\limits_{n\in\mathbb{Z}}(a_{n}L_{n}+b_{n}G_{n}),
\end{eqnarray}where $a_n, b_n\in\mathbb C$ for any $n\in\mathbb Z$.

Note that $$\mathbb{Z}=\bar0\cup\bar1\cup \cdots \cup \overline{|m|-1}.$$
Then (\ref{svir2}) can be rewritten as follows:
\begin{eqnarray}\label{svir666}
\Delta(L_{m})
&=&\sum\limits_{\overline{i}\in E}\sum\limits_{k=s_{i}}^{t_{i}}a_{i+km}L_{i+km}+\sum\limits_{\overline{i}\in F}\sum\limits_{k=p_{i}}^{q_{i}}b_{i+km}G_{i+km},
\end{eqnarray}
where $E, F\subset\mathbb{Z}_m$, and $s_{i}\leq t_{i}, p_{i}\leq q_{i}\in\mathbb{Z}$ for any $i\in\mathbb Z$.

For $L_{m}+x L_{0}, x\in\mathbb{C}^{*}$, since $\Delta$ is a local derivation, there exists
$\sum\limits_{n\in \mathbb{Z}}(a_{n}^{\prime}L_{n}+b_{n}^{\prime}G_{n})\in \mathcal{S}$,
where $a_{n}^{\prime},b_{n}^{\prime}\in\mathbb{C}$ for any $n\in\mathbb Z$, such that
\begin{eqnarray}\label{svir4}
\nonumber\Delta( L_{m})&=&\nonumber\Delta(L_{m}+xL_{0})\\\nonumber
&=&\nonumber[\sum\limits_{n\in \mathbb{Z}}(a_{n}^{\prime}L_{n}+b_{n}^{\prime}G_{n}), L_{m}+x L_{0}]\\\nonumber
&=&\sum\limits_{\overline{i}\in E}\sum\limits_{k=s_{i}^{\prime}}^{t_{i}^{\prime}+1}((i+(k-2)m)a_{i+(k-1)m}^{\prime}+x(i+km)a_{i+km}^{\prime})L_{i+km}\\
&+&\sum\limits_{\overline{i}\in F}\sum\limits_{k=p_{i}^{\prime}}^{q_{i}^{\prime}+1}(\frac{1}{2}(2i+(2k-3)m)b_{i+(k-1)m}^{\prime}+x(i+km)b_{i+km}^{\prime})G_{i+km},
\end{eqnarray}
where $a_{i+(s_{i}^{\prime}-1)m}^{\prime}=a_{i+(t_{i}^{\prime}+1)m}^{\prime}=b_{i+(p_{i}^{\prime}-1)m}^{\prime}=b_{i+(q_{i}^{\prime}+1)m}^{\prime}
=0$.
Note that we have the same $E$ and $F$ in (\ref{svir666}) and (\ref{svir4}).

\begin{lemma}\label{mainlem}
Let $\Delta$ be a local derivation on $\mathcal{S}$ such that $\Delta(L_{0})=0$. Then $E=F=\bar0$ in $(\ref{svir666})$ and $(\ref{svir4})$.
\end{lemma}
\begin{proof} It is essentially same as that of Lemma 3.1 in \cite{WGL} (also see Lemma 3.2 in \cite{CZZ}).
\end{proof}

\begin{lemma}\label{mainlem1}
Let $\Delta$ be a local derivation on $\mathcal{S}$ such that $\Delta(L_{0})=0$. Then for any  $m\in\mathbb{Z}^*$,
\begin{eqnarray*}
\Delta(L_{m})=a_{m}L_{m}+b_{m}G_{m}
\end{eqnarray*} for some $a_m, b_m\in\mathbb C$.
\end{lemma}

\begin{proof}

By Lemma \ref{mainlem}, (\ref{svir666}) and (\ref{svir4}), we have
\begin{eqnarray*}\label{wqy2}
&&\sum\limits_{k=s}^{t}a_{km}L_{km}=\sum\limits_{k=s'}^{t'+1}((k-2)ma_{(k-1)m}^{\prime}+xkma_{km}^{\prime})L_{km}\label{wqy2},\\
&& \sum\limits_{k=p}^{q}b_{km}G_{km}=\sum\limits_{k=p'}^{q'+1}(\frac{1}{2}(2k-3)mb_{(k-1)m}^{\prime}+xkmb_{km}^{\prime})G_{km}\label{wqy22},
\end{eqnarray*}
where $s=s_{0},t=t_{0},s^{\prime}=s_{0}^{\prime},t^{\prime}=t_{0}^{\prime},p=p_{0},q=q_{0},p^{\prime}=p_{0}^{\prime},q^{\prime}=q_{0}^{\prime}$ and we have assigned $a_{(s^{\prime}-1)m}^{\prime}=a_{(t^{\prime}+1)m}^{\prime}=b_{(p^{\prime}-1)m}^{\prime}=b_{(q^{\prime}+1)m}^{\prime}
=0$. We may assume that $a _{sm}, a_{tm}, b_{pm}, b_{qm}, a_{s^{\prime}m}^{\prime}$,  $a_{t^{\prime}m}^{\prime}$, $b_{p^{\prime}m}^{\prime}$, $b_{q^{\prime}m}^{\prime}\neq0$. Clearly $s^{\prime}\leq s\leq t\leq t^{\prime}+1,p^{\prime}\leq p\leq q\leq q^{\prime}+1$, and $s^{\prime}=s$ if $s^{\prime}\neq0$, and $p^{\prime}=p$ if $p^{\prime}\neq0$.

 By the same considerations as that of Lemma 3.2 in \cite{WGL} (also see \cite{CZZ}) we can get $s=t=p=q=1$. The lemma follows.
\end{proof}

\begin{lemma}\label{mainlem2}
Let $\Delta$ be a local derivation on $\mathcal{S}$ such that $\Delta(L_{0})=\Delta(L_{1})=0$. Then $\Delta( L_{m})=0$ for any $m\in\mathbb{Z}$.
\end{lemma}
\begin{proof}
If $m\geq2$, by Lemma \ref{mainlem1}, there exist $c_{m},d_{m}\in\mathbb{C}$ and
\begin{eqnarray*}
\sum\limits_{i\in I}a'_iL_{i}+\sum\limits_{j\in J}b'_{j}G_{j}\in {\mathcal S},
\end{eqnarray*}
where $a_i', b_j'\in\mathbb C$ for any $i\in I, j\in J$, and $I, J$ are finite subsets of $\mathbb Z$, such that
\begin{eqnarray}\label{w1}
c_{m}L_{m}+d_{m}G_{m}\nonumber&=&\Delta( L_{m})=\nonumber\Delta( L_{m}+ L_{1})\\\nonumber
&=&\nonumber[\sum\limits_{i\in I}a'_iL_{i}+\sum\limits_{j\in J}b'_{j}G_{j}, L_{m}+L_{1}].
\end{eqnarray}
Clearly $I, J\subset\{0, 1, m\}$. By easy calculations we have $c_{m}=d_m=0$. Similarly, if $m<0$, we can also get $c_{m}=d_{m}=0$. The proof is completed.
\end{proof}

%\section{Local derivation on the odd part of $\mathcal{S}$}

Now we shall determine $\Delta(G_m)$ for any local derivations $\Delta$ of $\mathcal{S}$.

 Let $\Delta$ be a local derivation on $\mathcal{S}$ with $\Delta( G_{0})=0 $. For $G_{m}$ with $m\neq0$, set
 \begin{eqnarray}\label{svir256}
\Delta(G_{m})&=&\sum\limits_{k\in\mathbb{Z}}(b_{k}L_{k}+a_{k}G_{k}),
\end{eqnarray} where $a_k, b_k\in\mathbb C$ for any $k\in\mathbb Z$.

By (\ref{svir256}) we see that there is a finite subset $F$ of $\mathbb{Z}_m$ such that
\begin{eqnarray}\label{svir166657}
\Delta(G_{m})
&=&\sum\limits_{k=p}^{q}b_{k}L_{k}+\sum\limits_{\overline{i}\in F}\sum\limits_{k=s_{i}}^{t_{i}}a_{i+km}G_{i+km},
\end{eqnarray}
where $s_{i}\leq t_{i}, p\leq q\in\mathbb{Z}$.

For $G_{m}+x G_{0}, x\in\mathbb{C}^{*}$,  there exists
$\sum\limits_{\overline{i}\in F}\sum\limits_{k=s_{i}^{\prime}}^{t_{i}^{\prime}}a_{i+km}^{\prime}L_{i+km}+\sum\limits_{k=p^{\prime}}^{q^{\prime}}b_{k}^{\prime}G_k\in \mathcal{S}$,
where $a_{i+km}^{\prime},b_{k}^{\prime}\in\mathbb{C}$, such that
\begin{eqnarray}\label{svir458}
\nonumber\Delta( G_{m})&=&\nonumber\Delta(G_{m}+xG_{0})\\\nonumber
&=&\nonumber[\sum\limits_{\overline{i}\in F}\sum\limits_{k=s_{i}^{\prime}}^{t_{i}^{\prime}}a_{i+km}^{\prime}L_{i+km}+\sum\limits_{k=p^{\prime}}^{q^{\prime}}b_{k}^{\prime}G_k, G_{m}+x G_{0}]\\\nonumber
&=&\frac{1}{2}\sum\limits_{\overline{i}\in F}\sum\limits_{k=s_{i}^{\prime}}^{t_{i}^{\prime}+1}((i+(k-3)m)a_{i+(k-1)m}^{\prime}+x(i+km)a_{i+km}^{\prime})G_{i+km}\\
&&+2\sum\limits_{k=p^{\prime}}^{q^{\prime}}(b_{k}^{\prime}L_{k+m}+xb_{k}^{\prime}L_{k}),
\end{eqnarray}
where $a_{i+(s_{i}^{\prime}-1)m}^{\prime}=a_{i+(t_{i}^{\prime}+1)m}^{\prime}=0$.
Note that we have the same $F$  in (\ref{svir166657}) and (\ref{svir458}).

\begin{lemma}\label{mainlemS}
Let $\Delta$ be a local derivation on $\mathcal{S}$ such that $\Delta(G_{0})=0$. Then $b_k=0$ for any $k\in\mathbb Z$ and $F=\bar0$ in $(\ref{svir166657})$.
\end{lemma}
\begin{proof}
Without losing generality, we can suppose that $m\ge 1$ and $b_p, b_q, b_p', b_q'\ne0$.
For the first statement,  comparing the right hand sides of (\ref{svir166657}) and (\ref{svir458}), we see that
$$\sum\limits_{k=p}^{q}b_{k}L_{k}=2\sum\limits_{k=p^{\prime}}^{q^{\prime}}(b_{k}^{\prime}L_{k+m}+xb_{k}^{\prime}L_{k}).$$
So $p=p'$ and
\begin{eqnarray*}
&&b_p=2xb_p',\\
&&b_{p+m}=2(b_p'+xb_{p+m}'),\\
&&b_{p+2m}=2(b_{p+m}'+xb_{p+2m}'),\\
&&\cdots\cdots\\
&&b_{p+lm}=2b_{p+(l-1)m}',
\end{eqnarray*}
for some $l\in\mathbb Z_+.$

Since $b_p\ne 0$, eliminating $b_{p}',\cdots,b_{p+lm}'$ in this order by substitution we see that
\begin{equation}\label{odd11}
b_p+\ast x+\ast x^2+\cdots+\ast x^l=0,
\end{equation} where $\ast$ are constants which are independent of $x$.
We can always find some $x\in\mathbb{C}^{\ast}$ not satisfying (\ref{odd11}), which is a contradiction. The first statement  follows.

The second  statement is essentially same as that of Lemma \ref{mainlem} (also see Lemma 3.2 in \cite{CZZ}, Lemma 3.1 in \cite{WGL}). The lemma follows.\end{proof}

\begin{lemma}\label{mainlemS1}
Let $\Delta$ be a local derivation on $\mathcal{S}$ such that $\Delta(G_{0})=0$. Then for any $m\in\mathbb Z^*$,
\begin{eqnarray*}
\Delta(G_{m})=a_{m}G_{m}
\end{eqnarray*} for some $a_m\in\mathbb C$.
\end{lemma}
\begin{proof}
By Lemma \ref{mainlemS}, we can suppose that
\begin{eqnarray}\label{svir66657}
\Delta(G_{m})=\sum\limits_{k=s}^ta_{km}G_{km},
\end{eqnarray}
where $a_{km}\in\mathbb C$ and $s\le t\in\mathbb{Z}$.

Moreover, by (\ref{svir458}) and (\ref{svir66657}),  we have
\begin{eqnarray}\label{wqy212}
\sum\limits_{k=s}^{t}a_{km}G_{km}&=&\frac{1}{2}\sum\limits_{k=s^{\prime}}^{t^{\prime}+1}((k-3)ma_{(k-1)m}^{\prime}+xkma_{km}^{\prime})G_{km}, \label{wqy2122}
\end{eqnarray}
where $s=s_{0},t=t_{0},s^{\prime}=s_{0}^{\prime},t^{\prime}=t_{0}^{\prime}$ and we have assigned $a_{(s^{\prime}-1)m}^{\prime}=a_{(t^{\prime}+1)m}^{\prime}=0$. We may assume that $a _{sm}, a_{tm}, a_{s^{\prime}m}^{\prime}, a_{t^{\prime}m}^{\prime}\neq0$. Clearly $s^{\prime}\leq s\leq t\leq t^{\prime}+1$, and $s^{\prime}=s$ if $s^{\prime}\neq0$.

\noindent{\bf Claim 1}: $t=1$

Similar to the proof of Lemma 3.2 in \cite{WGL} we can get $s'\geq0,s\geq1$. Therefore, we know that $t\geq s\geq1$. We need  to prove that $t=1$ in (\ref{wqy212}). Otherwise we assume that $t>1$, and then $t^{\prime}>0$.

\noindent{\bf Case 1}: $t^{\prime}>2$.

In this case we can show that $t=t^{\prime}+1$ as in the above arguments. If $s^{\prime}\geq1$, we see that $s=s^{\prime}$ and $s<t$. From (\ref{wqy212}) we obtain a set of (at least two) equations
\begin{eqnarray}\label{wqy612}
2a_{sm}&=&xpma_{sm}^{\prime}\nonumber ;\\\nonumber
2a_{(s+1)m}&=&(s-2)ma_{sm}^{\prime}+x(s+1)ma_{(s+1)m}^{\prime};\\\nonumber
&\vdots&\\
2a_{tm}&=&(t-3)ma_{(t-1)m}^{\prime}.
\end{eqnarray}
Using the same arguments as Lemma \ref{mainlemS}, the equation (\ref{wqy612}) makes contradictions. So $s^{\prime}=0$. Now we have
\begin{eqnarray*}
s^{\prime}=0,s\geq1,t=t^{\prime}+1>3.
\end{eqnarray*}
By (\ref{wqy212}) we obtain a set of (at least two) equations
\begin{eqnarray}\label{wqy912}
2a_{m}&=&xma_{m}^{\prime}\nonumber ;\\\nonumber
2a_{2m}&=&-ma_{m}^{\prime}+2xa_{2m}^{\prime};\\\nonumber
&\vdots&\\
2a_{tm}&=&(t-3)ma_{(t-1)m}^{\prime}.
\end{eqnarray}
Using the same arguments again, (\ref{wqy912}) makes contradictions. So $t^{\prime}=1$ or $2$.  Now we have

\noindent{\bf Case 2}: $t^{\prime}=2$.

If $t^{\prime}=2$ we see that $t<t^{\prime}+1=3$, then the coefficient of $G_{3m}$ is zero.  It also gets a contradiction.

\noindent{\bf Case 3}: $t^{\prime}=1$.

In this case $t=2$, and then
 \begin{eqnarray*}\label{wqh91}
2a_{m}&=&xma_{m}^{\prime}\nonumber ;\\ \nonumber
2a_{2m}&=&-ma_{m}^{\prime}.
\end{eqnarray*}
 Using the same arguments again, which also gets a contradiction.

Combining with Cases 1-3, we get that $s=t=1$. The lemma follows.
\end{proof}

\begin{lemma}\label{mainlemS2}
Let $\Delta$ be a local derivation on $\mathcal{S}$ such that $\Delta(G_{0})=\Delta(G_{1})=0$. Then $\Delta( G_{m})=0$ for any $m\in\mathbb{Z}$.
\end{lemma}
\begin{proof}
If $m\geq2$, by Lemma \ref{mainlemS1}, there exist $d_{m}\in\mathbb{C}$ and
\begin{eqnarray*}
\sum\limits_{i\in I}a'_iL_{i}+\sum\limits_{j\in J}b'_{j}G_{j}\in {\mathcal S},
\end{eqnarray*}
where $a_i', b_j'\in\mathbb C$ for any $i\in I, j\in J$, and $I, J$ are finite subsets of $\mathbb Z$, such that
\begin{eqnarray}\label{w1}
d_{m}G_{m}\nonumber&=&\Delta( G_{m})=\nonumber\Delta( G_{m}+ G_{1})\\\nonumber
&=&\nonumber[\sum\limits_{i\in I}a'_iL_{i}+\sum\limits_{j\in J}b'_{j}G_{j}, G_{m}+G_{1}].
\end{eqnarray}
Clearly $I\subset\{0, m-1\}$ and $J=\emptyset$. By simple calculations we have $d_m=0$. Similarly, if $m<0$, we can also get $c_{m}=d_{m}=0$. The proof is completed.
\end{proof}

\section{Local derivation on the super Virasoro algebras}

\begin{lemma}\label{mainlemS3}
Let $\Delta$ be a local derivation on $\mathcal{S}$ such that $\Delta(L_{0})=\Delta(L_1)=0$. Then $\Delta( G_0)=0$.
\end{lemma}
\begin{proof}
Suppose that
\begin{eqnarray*}
\Delta(G_0)=\sum_{i\in I}a_iL_i+\sum_{j\in J}b_jG_j, \label{supposes1}
\end{eqnarray*}
where $a_i, b_j\in\mathbb C^*$ for any $i\in I, j\in J$, and $I, J$ are finite subsets of $\mathbb Z$.

Now $\Delta(G_0)=\Delta(G_0+xL_0)=[a_x, G_0+xL_0]$ for some $a_x=\sum_{i\in I'}a_i'L_i+\sum_{j\in J'}b_j'G_j\in \mathcal{S}$, where $a_i', b_j'\in\mathbb C$ for any $i\in I', j\in J'$, and $I', J'$ are finite subsets of $\mathbb Z$. Then
\begin{eqnarray*}
\Delta(G_0)&=&\sum_{i\in I}a_iL_i+\sum_{j\in J}b_jG_j\\
&=&\Delta(G_0+xL_0)\\&=&\sum_{j\in J'}2b_j'L_j+\sum_{i\in I'}ixa_i'L_i+\frac12\sum_{i\in I'}ia_i'G_i+\sum_{j\in J'}jxb_j'G_j.
\end{eqnarray*}

So $I=J$ and
\begin{eqnarray}\label{eqs2}
a_i=ixa_i'+2b_i', \ \
b_i=\frac12ia_i'+ixb_i'\end{eqnarray}
 for any $i\in I$.

Then we have $a_i-2xb_i=2(1-ix^2)b_i'$.
For $i\ne 0$,  choosing $x=\pm \sqrt{i^{-1}}$ for $i\in I$ repeatedly, we can get $a_i\pm \sqrt{i^{-1}} 2b_i=0$  for any $i\in I$. Then $a_i=b_i=0$ for any $0\ne i\in I$.

For $i=0$, we have $b_0=0$ by \eqref{eqs2}. So $\Delta(G_0)=a_0L_0$ for some $a_0\in\mathbb C$.

Using $a_0L_0=\Delta(G_0)=\Delta(G_0+L_1)=[u, G_0+L_1]$ for some $u\in S$, we can get $a_0=0$ by some easy calculations as above.
\end{proof}

\begin{lemma}\label{mainlemS4}
Let $\Delta$ be a local derivation on $\mathcal{S}$ such that $\Delta(L_1)=0$ and $\Delta(G_1)=aG_1$ for some $a\in\mathbb C$. Then $\Delta( G_1)=0$.
\end{lemma}
\begin{proof}
For any $x\in\mathbb C^*$, there exists $a_x=\sum_{i\in I} a'_iL_i+\sum_{j\in J} b'_jG_j\in \mathcal S$, where $a'_i, b'_j\in\mathbb C^*$ for any $i\in I, j\in J$,  such that
$\Delta(G_1)=\Delta(G_1+xL_1)=[a_x, G_1+xL_1]$. Then
\begin{eqnarray*}
\Delta(G_1)&=&aG_1=\Delta(G_1+xL_1)\\&=&\sum_{j\in J} 2b_j'L_{j+1}+x\sum_{i\in I} (i-1)a_i'L_{i+1}+\sum_{i\in I} (\frac12i-1)a_i'G_{i+1}+x\sum_{j\in J} (j-\frac12)b_j'G_{j+1}.
\end{eqnarray*}
So we have $I=J$ and
\begin{eqnarray}\label{eqs3}
a=-a_0'-\frac12b_0'x, \  \
0=2b_0'-a_0'x.\end{eqnarray}
Choosing $x=2\sqrt{-1}$ in \eqref{eqs3}, we can get $a=0$.
\end{proof}

Now we can get our main result in this paper.
\begin{theorem}\label{main}
Every local derivation
on  the centerless super Virasoro algebra $\mathcal{S}$ is a derivation.
\end{theorem}
\begin{proof}
Let $\Delta$ be a local derivation on $\mathcal{S}$.  There exists $y\in \mathcal{S}$ such that $\Delta(L_{0})=[y, L_{0}]$. Replaced $\Delta$ by $\Delta-\mathrm{ad}(y)$,  we can get $\Delta(L_{0})=0$.

By Lemma \ref{mainlem1}, there exist $c,d\in\mathbb{C}$ such that
\begin{eqnarray*}
\Delta(L_{1})=cL_{1}+dG_1.
\end{eqnarray*}
Replaced $\Delta$ by $\Delta+c \mathrm{ad} L_{0}+2d {\rm ad} G_0$, we have
\begin{eqnarray*}
\Delta(L_{0})=0,\Delta(L_{1})=0.
\end{eqnarray*}
By Lemma \ref{mainlem2},  we have
\begin{eqnarray*}
\Delta(L_{m})=0,\forall m\in\mathbb{Z}.
\end{eqnarray*}

By Lemma \ref{mainlemS3}, we get $\Delta(G_0)=0$.
So by Lemma \ref{mainlemS1} we can suppose that $\Delta(G_1)=aG_1$. By Lemma \ref{mainlemS4}, we get $\Delta(G_1)=0$.
By Lemma \ref{mainlemS2}, we get $\Delta(G_m)=0$ for all $m\in\mathbb Z$.
 The proof is completed.
\end{proof}

\begin{corollary}
Every local derivation on  the super Virasoro algebra ${\rm SVir}[0]$ is a derivation.
\end{corollary}

\begin{remark}

With the same calculations as above, we can also prove that every local derivation on  the super Virasoro algebra SVir$[\frac12]$ is a derivation.

\end{remark}

\section{Local and 2-Local Automorphisms on the super Virasoro algebras}

In this section we present the main result on local and 2-local automorphisms of the super Virasoro algebras of ${\rm SVir}[\epsilon],\epsilon\in\{0,\frac{1}{2}\}$. Denoted by $\hat{S}:={\rm SVir}[\epsilon]$.

\begin{lemma}\label{lem-local-auto2}
Suppose $\sigma\in {\rm Aut}\,(\hat{S})$ satisfying $\sigma(L_{1})=L_{1}$, then $\sigma=id_{\hat{S}}$.
\end{lemma}

\begin{proof}
Since $\sigma(L_{1})=L_{1}$,  by \eqref{aut1}, we have
\begin{eqnarray*}
\sigma(L_{1})=\varepsilon aL_{\varepsilon},
\end{eqnarray*}
and thus we can prove that $\varepsilon=1,a=1$. By Lemma \ref{lem-local-auto1}, and thus
$\sigma=id_{\hat{S}}$.
\end{proof}

\begin{theorem}\label{mainaut1}
Every  2-local automorphism
on  the super Virasoro algebra $\hat{S}$ is an automorphism.
\end{theorem}
\begin{proof}
Suppose that $\phi$ is an  2-local automorphism of $\hat{S}$. For $L_{1}$, there exists an automorphism $\sigma_{L_{1},L_{1}}$ on $\hat{S}$ such that $\phi(L_{1})=\sigma_{L_{1},L_{1}}(L_{1})$. Let $\phi'=\sigma^{-1}_{L_{1},L_{1}}\circ \phi$, then $\phi'$ is a 2-local automorphism of $\hat{S}$ such that $\phi'(L_{1})=L_{1}$. For any $z\in\hat{S}$, there exists an automorphism $\theta_{L_{1},z}$ on $\hat{S}$ such that
$$L_{1}=\phi'(L_{1})=\theta_{L_{1},z}(L_{1}),\phi'(z)=\theta_{L_{1},z}(z).$$ It follows that $\theta_{L_{1},z}=id_{\hat{S}}$ by Lemma \ref{lem-local-auto2}, and thus $\phi'(z)=z$, i.e. $\phi'=id_{\hat{S}}$. Hence
$$\phi=\sigma_{L_{1},L_{1}}.$$ Therefore $\phi$ is an automorphism of $\hat{S}$.
\end{proof}

\begin{theorem}\label{mainaut2}
Every local automorphism
on  the super Virasoro algebra $\hat{S}$ is an automorphism.
\end{theorem}
\begin{proof}
Suppose that $\phi$ is a local automorphism of $\hat{S}$, For $L_{1}$, there exists an automorphism $\sigma_{L_{1}}$ on $\hat{S}$ such that $\phi(L_{1})=\sigma_{L_{1}}(L_{1})$. Let $\phi'=\sigma^{-1}_{L_{1}}\circ\phi$, then $\phi'$ is a local automorphism of $\hat{S}$ such that $\phi'(L_{1})=L_{1}$.

For any $G_{r}$ with $r\in \mathbb{Z}+\epsilon$, there exist automorphisms $\theta_{G_{r}}$ and  $\theta_{G_{r}+L_{1}}$ of $\hat{S}$ such that $$\phi'(G_{r})=\theta_{G_{r}}(G_{r}),\phi'(G_{r}+L_{1})=\theta_{G_{r}+L_{1}}(G_{r}+L_{1}).$$
Then we have
\begin{eqnarray}\label{aut5}
\theta_{G_{r}}(G_{r})+L_{1}&=&\phi'(G_{r})+L_{1}=\phi'(G_{r})+\phi'(L_{1})\nonumber\\
&=&\phi'(G_{r}+L_{1})=\theta_{G_{r}+L_{1}}(G_{r}+L_{1})\nonumber\\
&=&\theta_{G_{r}+L_{1}}(G_{r})+\theta_{G_{r}+L_{1}}(L_{1}).\nonumber
\end{eqnarray}
 By Lemma \ref{lem-local-auto1} we deduce that $\theta_{G_{r}+L_{1}}(L_{1})=L_{1}$. Therefore $\theta_{G_{r}+L_{1}}=id_{\hat{S}}$ by Lemma \ref{lem-local-auto2}, and thus $\phi'(G_{r})=\theta_{G_{r}+L_{1}}(G_{r})=G_{r}, r\in \mathbb{Z}+\epsilon.$

Similarly, by substituting $G_{r}$ with $C$ or $L_{m}$ where $m\neq\pm1$, we deduce that $\phi'(C)=C$ and $\phi'(L_{m})=L_{m}, m\neq\pm1.$

 Next, we consider the case when $m=-1$. There exist automorphisms $\theta_{L_{-1}}$ and  $\theta_{L_{-1}+L_{1}}$ of $\hat{S}$ such that $$\phi'(L_{-1})=\theta_{L_{-1}}(L_{-1}),\phi'(L_{-1}+L_{1})=\theta_{L_{-1}+L_{1}}(L_{-1}+L_{1}).$$
Then we have
 \begin{eqnarray}\label{aut6}
\theta_{L_{-1}}(L_{-1})+L_{1}&=&\phi'(L_{-1})+L_{1}=\phi'(L_{-1})+\phi'(L_{1})\nonumber\\
&=&\phi'(L_{-1}+L_{1})=\theta_{L_{-1}+L_{1}}(L_{-1}+L_{1})\nonumber\\
&=&\theta_{L_{-1}+L_{1}}(L_{-1})+\theta_{L_{-1}+L_{1}}(L_{1}).\nonumber
\end{eqnarray}
 By Lemma \ref{lem-local-auto1}, if
 \begin{eqnarray*}
\theta_{L_{-1}+L_{1}}(L_{1})=L_{1},
\end{eqnarray*}
 then $\theta_{L_{-1}+L_{1}}=id_{\hat{S}}$, so $\phi'(L_{-1})=\theta_{L_{-1}+L_{1}}(L_{-1})=L_{-1}$.
If
 \begin{eqnarray*}
\theta_{L_{-1}+L_{1}}(L_{-1})=L_{1},
\end{eqnarray*}
 then $\theta_{L_{-1}+L_{1}}(L_{1})=L_{-1}$ by Lemma \ref{lem-local-auto1}, so $\phi'(L_{-1})=\theta_{L_{-1}+L_{1}}(L_{1})=L_{-1}$.

Therefore
\begin{eqnarray*}
\phi'(L_{m})=L_{m},m\in\mathbb{Z}.
\end{eqnarray*}

So $\phi'=id_{\hat{S}}$, and thus $\phi=\sigma_{L_{1}}$. Therefore $\phi$ is an automorphism.

\end{proof}

\section*{Acknowledgments}
This work is partially supported by the NNSF (Nos. 12071405, 11971315, 11871249), the Innovation Project of Excellent Doctoral Students of Xinjiang University, China (Grant No. XJU2023BS019).

\end{document}